\documentclass[a4paper,11pt]{amsart}
\usepackage{amssymb}
\usepackage{mathptmx}
\usepackage[all]{xy}

\title{Ideal-adic semi-continuity problem for \linebreak minimal log discrepancies}

\author{Masayuki Kawakita}
\address{Research Institute for Mathematical Sciences, Kyoto University, Kyoto 606-8502, Japan}
\email{masayuki@kurims.kyoto-u.ac.jp}

\theoremstyle{plain}
\newtheorem{theorem}{Theorem}[section]
\newtheorem{proposition}[theorem]{Proposition}
\newtheorem{lemma}[theorem]{Lemma}
\newtheorem{corollary}[theorem]{Corollary}
\newtheorem{conjecture}[theorem]{Conjecture}
\newtheorem{lemma-definition}[theorem]{Lemma-Definition}
\newtheorem*{theorem*}{Theorem}
\newtheorem*{conjecture*}{Conjecture}

\theoremstyle{definition}
\newtheorem{definition}[theorem]{Definition}

\theoremstyle{remark}
\newtheorem{remark}{Remark}[theorem]
\newtheorem{remark-example}[remark]{Remark-Example}
\newtheorem{example}[theorem]{Example}
\newtheorem*{acknowledgements}{Acknowledgements}

\newcommand{\bA}{\mathbb{A}}
\newcommand{\bL}{\mathbb{L}}
\newcommand{\bQ}{\mathbb{Q}}
\newcommand{\bR}{\mathbb{R}}
\newcommand{\bZ}{\mathbb{Z}}
\newcommand{\cB}{\mathcal{B}}
\newcommand{\cC}{\mathcal{C}}
\newcommand{\cD}{\mathcal{D}}
\newcommand{\cI}{\mathcal{I}}
\newcommand{\cJ}{\mathcal{J}}
\newcommand{\cL}{\mathcal{L}}
\newcommand{\cM}{\mathcal{M}}
\newcommand{\cO}{\mathcal{O}}
\newcommand{\fa}{\mathfrak{a}}
\newcommand{\fb}{\mathfrak{b}}
\newcommand{\fc}{\mathfrak{c}}
\newcommand{\ff}{\mathfrak{f}}
\newcommand{\fg}{\mathfrak{g}}
\newcommand{\fI}{\mathfrak{I}}
\newcommand{\fR}{\mathfrak{R}}

\DeclareMathOperator{\Hom}{\mathcal{H}\mathit{om}}
\DeclareMathOperator{\mld}{mld}
\DeclareMathOperator{\ord}{ord}
\DeclareMathOperator{\Supp}{Supp}

\begin{document}
\begin{abstract}
We discuss the ideal-adic semi-continuity problem for minimal log discrepancies by Musta\c{t}\u{a}. We study the purely log terminal case, and prove the semi-continuity of minimal log discrepancies when a Kawamata log terminal triple deforms in the ideal-adic topology.
\end{abstract}

\maketitle

\section*{Introduction}
In the minimal model program, singularities are measured in terms of log discrepancies. The log discrepancy is attached to each divisor on an extraction of the singularity, and their infimum is called the \textit{minimal log discrepancy}. Recently, de Fernex, Ein and Musta\c{t}\u{a} in \cite{dFEM10} after Koll\'ar in \cite{Kl08} proved the ideal-adic semi-continuity of log canonicity effectively to obtain Shokurov's ACC conjecture \cite{S96} for log canonical thresholds on l.c.i.\ varieties. This paper discusses its generalisation to minimal log discrepancies, proposed by Musta\c{t}\u{a}.

\begin{conjecture*}[Musta\c{t}\u{a}]
Let $(X,\Delta)$ be a pair, $Z$ a closed subset of $X$ and $\cI_Z$ its ideal sheaf. Let $\fa$ be an ideal sheaf and $r$ a positive real number. Then there exists an integer $l$ such that\textup{:} if an ideal sheaf $\fb$ satisfies $\fa+\cI_Z^l=\fb+\cI_Z^l$, then
\begin{align*}
\mld_Z(X,\Delta,\fa^r)=\mld_Z(X,\Delta,\fb^r).
\end{align*}
\end{conjecture*}

The $\mld$ above denotes the minimal log discrepancy. Musta\c{t}\u{a} observed that the conjecture on formal schemes implies the ACC for minimal log discrepancies on a fixed germ by the argument of generic limits of ideals.

The conjecture is not difficult to prove in the Kawamata log terminal case, stated in Theorem \ref{thm:klt}. It is however inevitable to deal with log canonical singularities in the study of limits. As its first extension, we treat a purely log terminal triple $(X,F+\Delta,\fa^r)$ with a Cartier divisor $F$ and control the minimal log discrepancy of $(X,G+\Delta,\fb^r)$ for $G,\fb$ close to $F,\fa$. Our main theorem compares minimal log discrepancies on $F,G$ rather than those on $X$. We adopt the weaker condition $\fa\approx_l\fb$ defined by $\fa^n+\cI_Z^{nl}=\fb^n+\cI_Z^{nl}$ for some $n$ to reflect the distance of $\fa,\fb$ with allowance of real exponents.

\begin{theorem*}[full form in Theorem \ref{thm:plt}]
$(X,\Delta)$, $Z$, $\fa$ and $r$ as in \textup{Conjecture}. Let $F$ be a reduced Cartier divisor such that $(X,F+\Delta,\fa^r)$ is plt about $Z$. Then there exists an integer $l$ such that\textup{:} if an effective Cartier divisor $G$ and an ideal sheaf $\fb$ satisfy $\cO_X(-F)\approx_l\cO_X(-G)$ and $\fa\approx_l\fb$, then $G$ is reduced about $Z$ and with its normalisation $\nu\colon G^\nu\to G$,
\begin{align*}
\mld_{F\cap Z}(F,\Delta_F,\fa^r\cO_F)=\mld_{\nu^{-1}(G\cap Z)}(G^\nu,\Delta_{G^\nu},\fb^r\cO_{G^\nu}).
\end{align*}
\end{theorem*}

The theorem can be regarded as an extension to the case when a variety as well as a boundary deforms, so it would provide a perspective in the study of the behaviour of minimal log discrepancies under deformations. It should be related to Shokurov's reduction \cite{S04} of the termination of flips. One can recover the equality $\mld_Z(X,F+\Delta,\fa^r)=\mld_Z(X,G+\Delta,\fb^r)$ if the precise inversion of adjunction in \cite{K+92} holds on $X$ such as l.c.i.\ varieties in \cite{EM04}, \cite{EMY03}.

We prove the theorem by using motivic integration by Kontsevich in \cite{Kn95} and Denef and Loeser in \cite{DL99}. Take a divisor $E$ on an extraction of $X$ whose restriction computes the minimal log discrepancy on $G$. By the plt assumption, the order of (the inverse image of) the Jacobian $\cJ'_G$ of $G$ along $E$ should be small in contrast to those of $F,G$, then it coincides with that of the Jacobian $\cJ'_F$ of $F$. This provides further the equality of the orders of the ideal sheaves $\cJ_{r,F},\cJ_{r,G}$, and we derive the theorem by the descriptions of minimal log discrepancies involving $\cJ_{r,F},\cJ_{r,G}$ by Ein, Musta\c{t}\u{a} and Yasuda in \cite{EMY03}.

We work over an algebraically closed field $k$ of characteristic zero throughout. $\bZ_{>0},\bZ_{\ge0},\bR_{>0},\bR_{\ge0}$ denote the sets of positive/non-negative, integers/real numbers.

\section{$\cI$-adic semi-continuity problem}
In this section we discuss general aspects of Musta\c{t}\u{a}'s $\cI$-adic semi-continuity problem for minimal log discrepancies.

For the study of limits, we formulate the notion of $\bR$-ideal sheaves by extending that of $\bQ$-ideal sheaves in \cite[Section 2]{K08}. On a scheme $X$ we let $\fR_X$ denote the free semi-group generated by the family $\fI_X$ of all ideal sheaves on $X$, with coefficients in the semi-group $\bR_{\ge0}$. An element of $\fR_X$ is written multiplicatively as $\fa_1^{r_1}\cdots\fa_k^{r_k}$ with $\fa_i\in\fI_X,r_i\in\bR_{\ge0}$. We say that $\fa,\fb\in\fR_X$ are \textit{adhered} if they are written as $\fa=\prod_{ij}\fa_{ij}^{r_im_{ij}}\cdot\cO_X^a\cdot0^{a'},\fb=\prod_{ik}\fb_{ik}^{r_in_{ik}}\cdot\cO_X^b\cdot0^{b'}$ in $\fR_X$ with $\fa_{ij},\fb_{ik}\in\fI_X$, $r_i,a,a',b,b'\in\bR_{\ge0}$, $m_{ij},n_{ik}\in\bZ_{\ge0}$, such that $\prod_j\fa_{ij}^{m_{ij}}$ equals $\prod_k\fb_{ik}^{n_{ik}}$ as ideal sheaves for each $i$, or $a',b'>0$. We say that $\fa,\fb\in\fR_X$ are \textit{equivalent} if there exist $\fc_0,\ldots,\fc_i\in\fR_X$ with $\fc_0=\fa,\fc_i=\fb$ such that each $\fc_{j-1}$ is adhered to $\fc_j$.

\begin{definition}
An \textit{$\bR$-ideal sheaf} on $X$ is an equivalence class of the above relation in $\fR_X$.
\end{definition}

We let $\fI_X^\bR$ denote the family of $\bR$-ideal sheaves on $X$. By an \textit{expression} of $\fa\in\fI_X^\bR$ we mean an element $\fa_1^{r_1}\cdots\fa_k^{r_k}\in\fR_X$ with $\fa_i\in\fI_X,r_i\in\bR_{>0}$ in the class of $\fa$.

\begin{remark}
While some literatures define an $\bR$-ideal sheaf as an element of $\fR_X$, we adopt that of $\fI_X^\bR$ from the viewpoint that for $\fa,\fb\in\fI_X$ one should identify for example the product of $\fa^{\sqrt{2}+1},\fb$ and that of $\fa^{\sqrt{2}},\fa\fb$, which remain different in $\fR_X$.
\end{remark}

\begin{remark}
Two ideal sheaves on a normal variety $X$ have the same order along every divisor if they have the same integral closure. We have an equivalence relation in $\fI_X$ by this. However we will not formulate in this direction, because the relation does not seem to be compatible with the notion of $\cI$-adic topology.
\end{remark}

One can extend the notions of orders and resolutions to $\bR$-ideal sheaves.

\begin{lemma-definition}\label{lem-def:Cartier}
Let $\ff_1^{r_1}\cdots\ff_k^{r_k}$, $\fg_1^{s_l}\cdots\fg_l^{s_l}$ be two expressions of the same $\bR$-ideal sheaf $\fa$ on a normal variety $X$. Suppose $\ff_i=\cO_X(-F_i)$ with a Cartier divisor $F_i$. Then $\fg_j=\cO_X(-G_j)$ with some Cartier divisor $G_j$, and $\sum_ir_iF_i=\sum_js_jG_j$. Such $\fa$ is called a locally principal $\bR$-ideal sheaf. In particular, the notion of resolutions of $\bR$-ideal sheaves makes sense.
\end{lemma-definition}

\begin{proof}
It suffices to prove that if the product $\fa_1\fa_2$ of ideal sheaves $\fa_1,\fa_2$ is locally principal, then so are $\fa_1,\fa_2$ also. Set $\fa_1\fa_2=\cO_X(-F)=f\cO_X$ locally. Then $F$ is decomposed into Weil divisors $F_1,F_2$ as $F=F_1+F_2$ such that $\fa_i\subset\cO_X(-F_i)$. On the other hand, one can write $f=\sum_jf_{1j}f_{2j}$ and $f_{1j}f_{2j}=c_jf$ with $f_{ij}\in\fa_i$, $c_j\in\cO_X$. Thus $1=\sum_jc_j$, so there exists $j$ such that $c_j$ is a unit, that is $f_{1j}f_{2j}\cO_X=\cO_X(-F)$. If we set $f_{ij}\cO_X=:\cO_X(-F'_i)$, then $F_i\le F'_i$ and $F=F_1+F_2=F'_1+F'_2$, so $\fa_i\subset\cO_X(-F_i)=\cO_X(-F'_i)\subset\fa_i$ which means $\fa_i=f_{ij}\cO_X$.
\end{proof}

We introduce the notion of $\cI$-adic topology for $\bR$-ideal sheaves.

\begin{definition}
Fix a closed subscheme $Z$ of a scheme $X$ and let $\cI_Z$ denote its ideal sheaf.
\begin{enumerate}
\item
For $\fa,\fb\in\fI_X$ and $l\in\bZ_{\ge0}$, we write $\fa\equiv_l\fb$ if
\begin{align*}
\fa+\cI_Z^l=\fb+\cI_Z^l.
\end{align*}
\item
For $\fa,\fb\in\fI_X$ and $l\in\bR$, we write $\fa\approx_l\fb$ if there exist $m\in\bZ_{\ge0},n\in\bZ_{>0}$ such that
\begin{align*}
\fa^n\equiv_m\fb^n,\qquad m/n\ge l.
\end{align*}
\item\label{itm:m-adic:sim}
For $\fa,\fb\in\fI_X^\bR$ and $l\in\bR$, we write $\fa\sim_l\fb$ if there exist expressions $\fa=\fa_1^{r_1}\cdots\fa_k^{r_k}$, $\fb=\fb_1^{r_1}\cdots\fb_k^{r_k}$ such that for each $i$
\begin{align*}
\fa_i\approx_{l/r_i}\fb_i.
\end{align*}
\end{enumerate}
\end{definition}

\begin{remark}\label{rmk:sim}
One may replace the condition $\fa_i\approx_{l/r_i}\fb_i$ in (\ref{itm:m-adic:sim}) above with $\fa_i\equiv_{l_i}\fb_i$, $l_i\ge l/r_i$.
\end{remark}

The following basic fact will be used repeatedly.

\begin{remark}\label{rmk:order}
{\itshape If $\fa\sim_l\fb$ and $l\ord_E\cI_Z>\ord_E\fa$ along a divisor $E$ on an extraction, then $\ord_E\fa=\ord_E\fb$}. This follows from the inequality $\ord_E\fa_i\le r_i^{-1}\ord_E\fa<r_i^{-1}l\ord_E\cI_Z\le\ord_E\cI_Z^{l_i}$ in the context $\fa_i+\cI_Z^{l_i}=\fb_i+\cI_Z^{l_i}$ of Remark \ref{rmk:sim}.
\end{remark}

We recall the theory of singularities in the minimal model program. A \textit{pair} $(X,\Delta)$ consists of a normal variety $X$ and a \textit{boundary} $\Delta$, that is an effective $\bR$-divisor such that $K_X+\Delta$ is an $\bR$-Cartier $\bR$-divisor. We treat a \textit{triple} $(X,\Delta,\fa)$ by attaching an $\bR$-ideal sheaf $\fa$. For a prime divisor $E$ on an extraction $\varphi\colon X'\to X$, that is proper and birational, its \textit{log discrepancy} is
\begin{align*}
a_E(X,\Delta,\fa):=1+\ord_E(K_{X'}-\varphi^*(K_X+\Delta))-\ord_E\fa.
\end{align*}
The image $\varphi(E)$ is called its \textit{centre} on $X$. $(X,\Delta,\fa)$ is said to be \textit{log canonical} (\textit{lc}), \textit{purely log terminal} (\textit{plt}), \textit{Kawamata log terminal} (\textit{klt}) respectively if $a_E(X,\Delta,\fa)\ge0$ ($\forall E$), $>0$ ($\forall$exceptional $E$), $>0$ ($\forall E$). For a closed subset $Z$ of $X$, the \textit{minimal log discrepancy}
\begin{align*}
\mld_Z(X,\Delta,\fa)
\end{align*}
over $Z$ is the infimum of $a_E(X,\Delta,\fa)$ for all $E$ with centre in $Z$. The log canonicity of $(X,\Delta,\fa)$ about $Z$ is equivalent to $\mld_Z(X,\Delta,\fa)\ge0$. See \cite[Section 1]{K09}, \cite{KM98} for details.

De Fernex, Ein and Musta\c{t}\u{a} in \cite{dFEM10} after Koll\'ar in \cite{Kl08} proved the $\cI$-adic semi-continuity of log canonicity effectively to obtain with \cite{dFM09} the ACC for log canonical thresholds on l.c.i.\ varieties. We state its direct extension to the case with boundaries here.

\begin{theorem}[{\cite[Theorem 1.4]{dFEM10}}]\label{thm:lct}
Let $(X,\Delta)$ be a pair and $Z$ a closed subset of $X$. Let $\fa$ be an $\bR$-ideal sheaf such that
\begin{align*}
\mld_Z(X,\Delta,\fa)=0.
\end{align*}
Then there exists a real number $l$ such that\textup{:} if an $\bR$-ideal sheaf $\fb$ satisfies $\fa\sim_l\fb$, then
\begin{align*}
\mld_Z(X,\Delta,\fb)=0.
\end{align*}
\end{theorem}

\begin{remark}\label{rmk:lct}
The $l$ is given effectively in terms of a divisor $E$ with centre in $Z$ such that $a_E(X,\Delta,\fa)=0$. One may take an arbitrary $l$ such that $l\ord_E\cI_Z>\ord_E\fa$ by Remark \ref{rmk:order}.
\end{remark}

We will consider its generalisation to minimal log discrepancies, proposed by Musta\c{t}\u{a}.

\begin{conjecture}[Musta\c{t}\u{a}]\label{cnj:mld}
Let $(X,\Delta)$ be a pair and $Z$ a closed subset of $X$. Let $\fa$ be an $\bR$-ideal sheaf. Then there exists a real number $l$ such that\textup{:} if an $\bR$-ideal sheaf $\fb$ satisfies $\fa\sim_l\fb$, then
\begin{align*}
\mld_Z(X,\Delta,\fa)=\mld_Z(X,\Delta,\fb).
\end{align*}
\end{conjecture}

This conjecture is related to Shokurov's ACC conjecture \cite{S88}, \cite[Conjecture 4.2]{S96} for minimal log discrepancies. In fact, Conjecture \ref{cnj:mld} has originated in Musta\c{t}\u{a}'s following observation parallel to \cite{dFEM10} by generic limits of ideals.

\begin{remark}[Musta\c{t}\u{a}]\label{rmk:acc}
{\itshape If Conjecture \textup{\ref{cnj:mld}} holds on formal schemes, then for a fixed pair $(X,\Delta)$, a closed point $x$ and a set $R$ of positive real numbers which satisfies the descending chain condition, the set
\begin{align*}
\{\mld_x(X,\Delta,\fa_1^{r_1}\cdots\fa_k^{r_k})\mid\fa_i\in\fI_X,r_i\in R\}
\end{align*}
satisfies the ascending chain condition.}

Indeed, we shall prove the stability of an arbitrary non-decreasing sequence of elements $c_i=\mld_x(X,\Delta,\fa_{i1}^{r_{i1}}\cdots\fa_{ik_i}^{r_{ik_i}})\ge0$. We may assume that $\fa_{ij}$ are non-trivial at $x$, then for a fixed divisor $F$ with centre $x$ we have $\sum_jr_{ij}\le\sum_jr_{ij}\ord_F\fa_{ij}\le a_F(X,\Delta)$. $R$ has its minimum $r$ say, whence $k_i\le r^{-1}a_F(X,\Delta)$. Thus by replacing with a subsequence, we may assume the constancy $k=k_i$. Further we may assume that $r_{ij}$ form a non-decreasing sequence for each $j$. Then $r_{ij}$ have a limit $r_j$ by $r_{ij}\le a_F(X,\Delta)$.

Take generic limits $\fa_j$ of $\fa_{ij}$ following \cite[Section 4]{dFEM10}, \cite{Kl08}. After extending the ground field $k$, we have $\fa_j$ on the completion $(\hat{X},\hat{\Delta})$ of $(X,\Delta)$ at $x$. Conjecture \ref{cnj:mld} on $(\hat{X},\hat{\Delta})$ provides an integer $i_0$ and a divisor $E$ on $X$ with centre $x$ such that for $i\ge i_0$, $\ord_{\hat{E}}\fa_j=\ord_E\fa_{ij}$ and
\begin{align*}
c:=\mld_{\hat{x}}(\hat{X},\hat{\Delta},\fa_1^{r_1}\cdots\fa_k^{r_k})&=a_{\hat{E}}(\hat{X},\hat{\Delta},\fa_1^{r_1}\cdots\fa_k^{r_k})\\
&=a_E(X,\Delta,\fa_{i1}^{r_1}\cdots\fa_{ik}^{r_k})=\mld_x(X,\Delta,\fa_{i1}^{r_1}\cdots\fa_{ik}^{r_k})\le c_i,
\end{align*}
with $\hat{x}:=x\times_X\hat{X}$, $\hat{E}:=E\times_X\hat{X}$. Hence
\begin{align*}
c\le c_i\le a_E(X,\Delta,\fa_{i1}^{r_{i1}}\cdots\fa_{ik}^{r_{ik}})=c+\sum_j(r_j-r_{ij})\ord_{\hat{E}}\fa_j,
\end{align*}
and its right-hand side converges to $c$. Thus $c_i=c$ for $i\ge i_0$.
\end{remark}

We expect an effective form of Conjecture \ref{cnj:mld}, but the naive generalisation of Remark \ref{rmk:lct} never holds.

\begin{remark-example}
Set $X=\bA^2$ with coordinates $x,y$ and $\fa=(x^2+y^3)\cO_X$, $\fb=x^2\cO_X$. The pair $(X,\fa^{2/3})$ has minimal log discrepancy $2/3=a_E(X,\fa^{2/3})$ over the origin $o$, computed by the divisor $E$ obtained by the blow-up at $o$. We have $\fa+\cI_o^3=\fb+\cI_o^3$ and $\ord_E\fa=2<3$, but $(X,\fb^{2/3})$ is not log canonical.
\end{remark-example}

We provide a few reductions of the conjecture.

\begin{remark}\label{rmk:inequality}
One inequality $\mld_Z(X,\Delta,\fa)\ge\mld_Z(X,\Delta,\fb)$ is obvious. For, take a divisor $E$ with centre in $Z$ such that $a_E(X,\Delta,\fa)=\mld_Z(X,\Delta,\fa)$, or negative in the non-lc case, and $l$ such that $l\ord_E\cI_Z>\ord_E\fa$ by Remark \ref{rmk:order}.
\end{remark}

\begin{remark}\label{rmk:BCHM}
Conjecture \ref{cnj:mld} is reduced to the case when {\itshape $X$ has $\bQ$-factorial terminal singularities, $\Delta$ is zero and $Z$ is irreducible}. Indeed, by \cite{BCHM10} one can construct an extraction $\varphi\colon X'\to X$ such that $X'$ has $\bQ$-factorial terminal singularities with effective $\Delta'$ defined by $K_{X'}+\Delta'=\varphi^*(K_X+\Delta)$. Then $\mld_Z(X,\Delta,\fa)=\mld_{\varphi^{-1}(Z)}(X',\Delta',\fa\cO_{X'})$, so the conjecture is reduced to that on $X'$. Further, we may assume $\Delta=0$ by forcing $\fa$ to absorb $\Delta$. It is obviously permissible to assume the irreducibility of $Z$.
\end{remark}

\begin{remark}
Mostly, we need just a weaker form of Conjecture \ref{cnj:mld} in which {\itshape an expression $\fa_1^{r_1}\cdots\fa_k^{r_k}$ of $\fa$ is fixed and only those $\fb=\fb_1^{r_1/n_1}\cdots\fb_k^{r_k/n_k}$ with $\fa_i^{n_i}\equiv_{l_i}\fb_i$, $l_i\ge ln_i/r_i$ are considered}. This is reduced to the case when {\itshape $\fa_i,\fb_i$ are locally principal $\bR$-ideal sheaves}. Indeed, after replacing $\fa_i^{r_i}$ with the $s$-uple of $\fa_i^{r_i/s}$ for some $s$, we may assume that $\mld_Z(X,\Delta,\fa)$ equals $\mld_Z(X,\Delta,\ff)$ locally for some $\ff=\prod_i(f_i\cO_X)^{r_i}$ with $f_i\in\fa_i$. By $\fa_i^{n_i}\equiv_{l_i}\fb_i$ one can write $f_i^{n_i}=g_i+h_i$ with $g_i\in\fb_i$, $h_i\in\cI_Z^{l_i}$, so $f_i^{n_i}\cO_X\equiv_{l_i}g_i\cO_X$. For $\fg=\prod_i(g_i\cO_X)^{r_i/n_i}$ the weaker conjecture for locally principal $\bR$-ideal sheaves provides
\begin{align*}
\mld_Z(X,\Delta,\fa)=\mld_Z(X,\Delta,\ff)=\mld_Z(X,\Delta,\fg)\le\mld_Z(X,\Delta,\fb),
\end{align*}
and we have the equality by Remark \ref{rmk:inequality}.
\end{remark}

In the klt case, it is not difficult to prove our conjecture.

\begin{theorem}\label{thm:klt}
Conjecture \textup{\ref{cnj:mld}} holds for a klt triple $(X,\Delta,\fa)$.
\end{theorem}

\begin{proof}
It suffices to prove $\mld_Z(X,\Delta,\fa)\le\mld_Z(X,\Delta,\fb)$ by Remark \ref{rmk:inequality}. As $(X,\Delta,\fa)$ is klt, we can fix $t,t'>0$ such that $\mld_Z(X,\Delta,\fa^{1+t}\cI_Z^{t'})=0$. Then by Theorem \ref{thm:lct} there exists
\begin{align*}
l\ge t^{-1}\mld_Z(X,\Delta,\fa)
\end{align*}
such that $\fa\sim_l\fb$ implies $\mld_Z(X,\Delta,\fb^{1+t}\cI_Z^{t'})=0$. Thus every divisor $E$ with centre in $Z$ satisfies
\begin{align*}
a_E(X,\Delta,\fb)>t\ord_E\fb.
\end{align*}
Suppose $a_E(X,\Delta,\fa)\neq a_E(X,\Delta,\fb)$, equivalently $\ord_E\fa\neq\ord_E\fb$. Then by Remark \ref{rmk:order},
\begin{align*}
\ord_E\fb\ge l\ord_E\cI_Z\ge l.
\end{align*}
The above three inequalities give $a_E(X,\Delta,\fb)>\mld_Z(X,\Delta,\fa)$, which completes the theorem.
\end{proof}

Even if we start with klt singularities, it is inevitable to deal with log canonical singularities in the study of limits of them.

\begin{example}
Set $X=\bA^2$ with coordinates $x,y$ and $\fa_n=x(x+y^n)\cO_X$. The limit of these $\fa_n$ is $\fa_\infty=x^2\cO_X$, so that of klt pairs $(X,\fa_n^{1/2})$ is a plt pair $(X,\fa_\infty^{1/2})=(X,x\cO_X)$.
\end{example}

It is standard to reduce to lower dimensions by the restriction of pairs to subvarieties. For a pair $(X,G+\Delta)$ such that $G$ is a reduced divisor which has no component in the support of effective $\Delta$, one can construct the \textit{different} $\Delta_{G^\nu}$ on its normalisation $\nu\colon G^\nu\to G$ as in \cite[Chapter 16]{K+92}, \cite[\S 3]{S92}. It is a boundary which satisfies the equality $K_{G^\nu}+\Delta_{G^\nu}=\nu^*((K_X+G+\Delta)|_G)$.

As the first extension of Theorem \ref{thm:klt}, we study the plt case in which the boundary involves a Cartier divisor $F$. Let $F$ be a Cartier divisor on a triple $(X,\Delta,\fa)$ such that $(X,F+\Delta,\fa)$ is plt. Then $F$ is normal by the connectedness lemma \cite[17.4 Theorem]{K+92}, \cite[5.7]{S92}, and the induced triple $(F,\Delta_F,\fa\cO_F)$ is klt. In this setting, we control $\mld_Z(X,G+\Delta,\fb)$ for $G,\fb$ close to $F,\fa$. We adopt the notation
\begin{align*}
F\sim_lG
\end{align*}
for the condition $\cO_X(-F)\sim_l\cO_X(-G)$, and $(F,\fa)\sim_l(G,\fb)$ for $F\sim_lG$, $\fa\sim_l\fb$. We compare minimal log discrepancies on $F,G$ rather than those on $X$, so $G$ should be a divisor of the following type.

\begin{definition}
A \textit{transversal} divisor on a triple $(X,\Delta,\fb)$ is a reduced Cartier divisor which has no component in the support of $\Delta$ or the zero locus of $\fb$.
\end{definition}

For example, an effective Cartier divisor $G$ is transversal if $(X,G+\Delta,\fb)$ is log canonical.

We state our theorem in the plt case, which will be proved in Section \ref{sec:plt}.

\begin{theorem}\label{thm:plt}
Let $(X,\Delta)$ be a pair and $Z$ a closed subset of $X$. Let $F$ be a reduced Cartier divisor and $\fa$ an $\bR$-ideal sheaf such that $(X,F+\Delta,\fa)$ is plt about $Z$. Then there exists a real number $l$ such that\textup{:} if an effective Cartier divisor $G$ and an $\bR$-ideal sheaf $\fb$ satisfy $(F,\fa)\sim_l(G,\fb)$, then $G$ is transversal on $(X,\Delta,\fb)$ about $Z$ and
\begin{align*}
\mld_{F\cap Z}(F,\Delta_F,\fa\cO_F)=\mld_{\nu^{-1}(G\cap Z)}(G^\nu,\Delta_{G^\nu},\fb\cO_{G^\nu}).
\end{align*}
\end{theorem}

Theorem \ref{thm:plt} compares minimal log discrepancies on different varieties, so it would provide a perspective in the study of their behaviour under deformations. One can interpret it as an extension of Theorem \ref{thm:klt} to the case when a variety as well as a boundary deforms. Theorem \ref{thm:plt} is also joined with Conjecture \ref{cnj:mld} via the \textit{precise inversion of adjunction} in \cite[Chapter 17]{K+92}.

\begin{conjecture}[precise inversion of adjunction]
Let $(X,G+\Delta)$ be a pair such that $G$ is a reduced divisor which has no component in the support of effective $\Delta$, and $Z$ a closed subset of $G$. Let $\Delta_{G^\nu}$ be the different on the normalisation $\nu\colon G^\nu\to G$. Then 
\begin{align*}
\mld_Z(X,G+\Delta)=\mld_{\nu^{-1}(Z)}(G,\Delta_{G^\nu}).
\end{align*}
\end{conjecture}

The equality of minimal log discrepancies on $X$ follows if the precise inversion of adjunction holds on $X$, such as l.c.i.\ varieties in \cite{EM04}, \cite{EMY03}.

\begin{corollary}
$(X,\Delta,\fa)$, $Z$ and $F$ as in Theorem \textup{\ref{thm:plt}}. Suppose that the precise inversion of adjunction holds on $X$. Then there exists a real number $l$ such that\textup{:} if effective Cartier divisors $G_i$ and an $\bR$-ideal sheaf $\fb$ satisfy $F\sim_lG_i$, $\fa\sim_l\fb$, then for $G=\sum_ig_iG_i$ with $1=\sum_ig_i$, $g_i\in\bR_{\ge0}$,
\begin{align*}
\mld_Z(X,F+\Delta,\fa)=\mld_Z(X,G+\Delta,\fb).
\end{align*}
\end{corollary}

\begin{proof}
We want $\mld_Z(X,F+\Delta,\fa)\le\mld_Z(X,G+\Delta,\fb)$ by Remark \ref{rmk:inequality}. Since $\mld_Z(X,G+\Delta,\fb)\ge\sum_ig_i\mld_Z(X,G_i+\Delta,\fb)$ by $K_X+G+\Delta=\sum_ig_i(K_X+G_i+\Delta)$, it is reduced to the case with a Cartier divisor $G$. We may assume $Z\subset F,G$ by Theorem \ref{thm:klt} and the argument after Lemma \ref{lem:approx}. Then the statement follows from Theorem \ref{thm:plt}. Note that the precise inversion of adjunction for triples is reduced to that for pairs.
\end{proof}

We close this section by one observation related to Conjecture \ref{cnj:mld}.

\begin{proposition}
Let $(X,\Delta)$ be a pair and $Z$ a closed subset of $X$. Let $\fa$ be an $\bR$-ideal sheaf. Then there exist real numbers $l$ and $0<t\le1$ such that\textup{:} if an $\bR$-ideal sheaf $\fb$ satisfies $\fa\sim_l\fb$, then
\begin{align*}
\mld_Z(X,\Delta,\fa)=\mld_Z(X,\Delta,\fa^{1-t}\fb^t).
\end{align*}
\end{proposition}

\begin{proof}
It suffices to prove $\mld_Z(X,\Delta,\fa)\le\mld_Z(X,\Delta,\fa^{1-t}\fb^t)$ by Remark \ref{rmk:inequality}. We may assume the log canonicity of $(X,\Delta,\fa)$. Fix a log resolution $\varphi\colon X'\to X$ of $(X,\Delta,\fa\cI_Z)$ and set $K_{X'}+\Delta':=\varphi^*(K_X+\Delta)$. Let $A$ denote the effective $\bR$-divisor on $X'$ defined by the locally principal $\bR$-ideal sheaf $\fa\cO_{X'}$, and $S$ the reduced divisor whose support is the union of the exceptional locus, $\Supp\Delta'$ and $\Supp A$. We take $0<t\le1$ such that $tA\le S$. By Theorem \ref{thm:lct} we have $l$ such that $\fa\sim_l\fb$ implies the log canonicity of $(X',S-tA,\fb^t\cO_{X'})$. In particular, for a divisor $E$ on an extraction $\psi\colon Y\to X'$ with $(\varphi\circ\psi)(E)\subset Z$,
\begin{align*}
a_E(X,\Delta,\fa^{1-t}\fb^t)&=a_E(X',(1-t)A,\fb^t\cO_{X'})-\ord_E\Delta'\\
&=a_E(X',S-tA,\fb^t\cO_{X'})+\ord_E(S-A-\Delta')\\
&\ge\ord_E(S-A-\Delta').
\end{align*}
$S-A-\Delta'=K_X'+S-(\varphi^*(K_X+\Delta)+A)\ge0$, and by a divisor $F$ with $\psi(E)\subset F \subset\varphi^{-1}(Z)$,
\begin{align*}
\ord_E(S-A-\Delta')&\ge\ord_F(S-A-\Delta')=a_F(X,\Delta,\fa).
\end{align*}
These two inequalities prove the proposition.
\end{proof}

\section{Purely log terminal case}\label{sec:plt}
The purpose of this section is to prove Theorem \ref{thm:plt}; see Lemmata \ref{lem:easy-plt} and \ref{lem:hard-plt}.

As $(X,\Delta)$ is klt, by \cite{BCHM10} there exists a $\bQ$-factorisation $\varphi\colon X'\to X$ which is isomorphic in codimension one. Then as in Remark \ref{rmk:BCHM} we can reduce the theorem to that on $X'$, and hence we may assume that {\itshape $X$ is $\bQ$-factorial and $\Delta=0$}. {\itshape We shall discuss on the germ at a closed point of $X$}.

We set the ideal sheaves in the context of motivic integration. Let $d$ denote the dimension of $X$. We fix a positive integer $r$ such that $rK_X$ is a Cartier divisor. We extend the construction in \cite[Section 2]{K08} to transversal divisors. A general l.c.i.\ subscheme $Y$ of dimension $d$ of a smooth ambient space $A$ which contains $X$ is the union
\begin{align}\label{eqn:lci}
Y=X\cup C^Y
\end{align}
of $X$ and another variety $C^Y$. The subscheme $D^Y:=C^Y|_X$ of $X$ is defined by the conductor ideal sheaf $\cC_{X/Y}:=\Hom_{\cO_Y}(\cO_X, \cO_Y)$, and is a divisor such that $\cO_X(rK_X)=\cO_X(-rD^Y)\omega_Y^{\otimes r}$. The summation $\cD'_X:=\sum_Y\cC_{X/Y}$ over all general $Y$ is called the \textit{l.c.i.\ defect ideal sheaf} of $X$, which one can define for reduced schemes of pure dimension. We treat the summation $\cD_{r,X}:=\sum_Y\cO_X(-rD^Y)$ also. For a reduced Cartier divisor $G$, the above $Y=X\cup C^Y$ has a Cartier divisor $Y_G=G\cup C^Y|_{Y_G}$. Thus $G$ has its l.c.i.\ defect ideal sheaf
\begin{align}\label{eqn:defect}
\cD'_G=\cD'_X\cO_G,
\end{align}
and we have $\cO_X(r(K_X+G))\cO_G=\cO_X(-rD^Y)\cO_G\cdot\omega_{Y_G}^{\otimes r}$.

Let $\cJ'_G$ be the Jacobian ideal sheaf of $G$, and $\cJ_{r,G}$ the image of the natural map $(\Omega_G^{d-1})^{\otimes r}\otimes\cO_X(-r(K_X+G))\to\cO_G$. Let $\tilde{\cJ}'_G,\tilde{\cJ}_{r,G}$ be the inverse images of them by the natural map $\cO_X\to\cO_G$. The argument in \cite{K08} provides the equality $\sum_Y{\cJ'_{Y_G}}^r\cO_G=\cJ_{r,G}\cdot\cD_{r,X}\cO_G$ similar to \cite[(2.4)]{K08} with the Jacobian $\cJ'_{Y_G}$ of $Y_G$. Its left-hand side is nothing but ${\cJ'_G}^r$. For, set local coordinates $x_1,\ldots,x_k$ of $A$ and the ideal sheaves $\cI_X,\cI_Y$ of $X,Y$ on $A$, and take $f_1,\ldots,f_c\in\cO_A$, $c=k-d+1$, such that $f_1|_X$ defines $G$ and $f_2,\ldots,f_c$ generate $\cI_Y$. Then for arbitrary $g_2,\ldots,g_c\in\cI_X$ and general $t_2,\ldots,t_c\in k$, the subscheme defined by $f_i+t_ig_i$, $2\le i\le c$, is a general l.c.i.\ $Y'$. Thus with $g_1:=f_1$ and $t_1\in k$, the $r$-th powers of determinants of $c\times c$ minors of the matrix $(\partial(f_i+t_ig_i)/\partial x_j)_{ij}|_G$ are contained in $\sum_Y{\cJ'_{Y_G}}^r\cO_G$, whence so are those of $(\partial g_i/\partial x_j)_{ij}|_G$. This means $\sum_Y{\cJ'_{Y_G}}^r\cO_G=\sum_{j\in{\cJ}'_G}j^r\cO_G$, and its right-hand side equals ${\cJ'_G}^r$ by the same trick. Hence we obtain
\begin{align}
\nonumber
{\cJ'_G}^r&=\cJ_{r,G}\cdot\cD_{r,X}\cO_G,\\
\label{eqn:defectX}
\tilde{\cJ}'_G{}^r+\cO_X(-G)&=\tilde{\cJ}_{r,G}\cdot\cD_{r,X}+\cO_X(-G).
\end{align}

We set
\begin{align*}
c:=\mld_{F\cap Z}(F,\fa\cO_F).
\end{align*}
As $(X,F,\fa)$ is plt, we can fix $t>0,t'\ge0$ such that
\begin{align*}
\mld_Z(X,F,\fa^{1+t}\tilde{\cJ}'_F{}^{rt}\cD'_X{}^t\cI_Z^{t'})=0.
\end{align*}

We will fix a log resolution $\bar{\varphi}\colon\bar{X}\to X$ of $(X,F,\fa\cI_Z\tilde{\cJ}'_F\tilde{\cJ}_{r,F}\cD'_X\cD_{r,X})$. Let $\bar{F}$ be the strict transform of $F$. By blowing up $\bar{X}$ further, we may assume the existence of a prime divisor $E_F\subset\bar{\varphi}^{-1}(F\cap Z)$ which intersects $\bar{F}$ properly and satisfies
\begin{align}\label{eqn:mldF}
a_{E_F}(X,F,\fa)=a_{E_F|_{\bar{F}}}(F,\fa\cO_F)=c.
\end{align}
Take the decomposition $\bar{\varphi}^*F=V_F+H_F$, where $V_F$ consists of prime divisors in $\bar{\varphi}^{-1}(Z)$ and $H_F$ those not in $\bar{\varphi}^{-1}(Z)$. By blowing up $\bar{X}$ further, we may assume that every divisor $\bar{E}$ with $\bar{E}\subset\Supp V_F$, $\bar{E}\cap\Supp H_F\neq\emptyset$ satisfies
\begin{align}\label{eqn:t^-1c}
\ord_{\bar{E}}V_F>t^{-1}c.
\end{align}
We take an integer $l_1$ such that
\begin{align}\label{eqn:l1}
l_1>\ord_{\bar{E}}V_F,\qquad l_1>\ord_{\bar{E}}\fa
\end{align}
for all divisors $\bar{E}$ on $\bar{X}$ with $\bar{\varphi}(\bar{E})\subset Z$. Note that
\begin{align}\label{eqn:estimate_l1}
l_1>t^{-1}c+1
\end{align}
unless $F\subset Z$.

The next lemma is a direct application of Theorem \ref{thm:lct} with Remark \ref{rmk:lct} by (\ref{eqn:l1}).

\begin{lemma}\label{lem:transversal}
For $\bR$-ideal sheaves $\fg,\fb$ such that $\cO_X(-F)\sim_{l_1}\fg$, $\fa\sim_{l_1}\fb$, we have $\mld_Z(X,\fg\fb^{1+t}\tilde{\cJ}'_F{}^{rt}\cD'_X{}^t\cI_Z^{t'})=0$. In particular if $(F,\fa)\sim_{l_1}(G,\fb)$ then $G$ is a transversal divisor on $(X,\fb)$.
\end{lemma}

We can replace the condition $F\sim_lG$ with the stronger one $F\approx_lG$ defined by $\cO_X(-F)\approx_l\cO_X(-G)$.

\begin{lemma}\label{lem:approx}
If $F\sim_lG$ with $l\ge l_1$, then $F\approx_lG$.
\end{lemma}

\begin{proof}
$G$ is reduced by Lemma \ref{lem:transversal}. By the definition of $F\sim_lG$ and Lemma-Definition \ref{lem-def:Cartier}, there exist decompositions $1=\sum_jf_jn_j$, $G=\sum_jf_jH_j$ with $f_j\in\bR_{>0}$, $n_j\in\bZ_{>0}$ and effective Cartier divisors $H_j$ such that $\cO_X(-n_jF)\equiv_{m_j}\cO_X(-H_j)$ with $m_j\ge l/f_j$. Note $\cO_X(-F)\approx_{m_j/n_j}\cO_X(-H_i)^{1/n_j}$ and $m_j/n_j\ge l/f_jn_j\ge l$. Hence all coefficients in $n_j^{-1}H_j$ are at most one by Lemma \ref{lem:transversal}. Thus each component $G_i$ of $G$ has $\ord_{G_i}H_j\le n_j$, so $1=\sum_jf_j\ord_{G_i}H_j\le \sum_jf_jn_j=1$ and $\ord_{G_i}H_j=n_j$, $H_j=n_jG$. Now the lemma follows from $\cO_X(-n_jF)\equiv_{m_j}\cO_X(-n_jG)$ and $m_j/n_j\ge l$.
\end{proof}

Now we may assume that {\itshape $Z$ is an irreducible proper subset of $F$, and is contained in $G$ also}. Indeed, since $F\approx_1G$ implies $F\cap Z=G\cap Z$ as sets, we may assume $Z\subset F,G$ by replacing $Z$ with $F\cap Z$. If $Z=F$ then $G\ge F$ and $F\approx_2G$ means $\cO_X(-nF)=\cO_X(-nF)(\cO_X(-n(G-F))+\cO_X(-nF))$ for some $n$, so $F=G$, $\fa\cO_F=\fb\cO_G$ and the statement is trivial. 

We write $(F,\fa)\approx_l(G,\fb)$ for the condition $F\approx_lG$, $\fa\sim_l\fb$. $G$ is transversal if $(F,\fa)\approx_{l_1}(G,\fb)$ by Lemma \ref{lem:transversal}. We then consider a log resolution $G'\to G$ embedded into some log resolution $\varphi\colon X'\to X$ of $(X,F+G,\fa\fb\tilde{\cJ}'_G\tilde{\cJ}_{r,G})$ which factors through $\bar{X}$. Set $\varphi'\colon X'\to\bar{X}$. Let $I$ denote the set of all $\varphi$-exceptional prime divisors $E$ on $X'$ intersecting $G'$, and $I_Z$ the subset of $I$ consisting of all $E$ with $\varphi(E)\subset Z$. By blowing up $X'$ further, we may assume that $G'$ does not intersect the strict transform of the divisorial part of the zero locus of $\fb$, and that for all $E\in I$
\begin{align}\label{eqn:resolutionG}
\varphi'(E)=\varphi'(E|_{G'}).
\end{align}
Then $\mld_{\nu^{-1}(Z)}(G^\nu,\fb\cO_{G^\nu})$ equals the minimum of $a_E(X,G,\fb)=a_{E|_{G'}}(G^\nu,\fb\cO_{G^\nu})$ for all $E\in I_Z$, or $-\infty$ if the minimum is negative.

\begin{lemma}\label{lem:l1}
If $(F,\fa)\approx_{l_1}(G,\fb)$, then for $E\in I_Z$
\begin{enumerate}
\item\label{itm:l1:ordJ'D}
$rt\ord_E\tilde{\cJ}'_F+t\ord_E\cD'_X+t\ord_E\fb\le a_{E|_{G'}}(G^\nu,\fb\cO_{G^\nu})$.
\item\label{itm:l1:ordFG}
$\ord_EF>t^{-1}c$ and $\ord_EG>t^{-1}c$.
\end{enumerate}
\end{lemma}

\begin{proof}
(\ref{itm:l1:ordJ'D}) \
It follows from Lemma \ref{lem:transversal}.

(\ref{itm:l1:ordFG}) \
If we write $\cI_Z\cO_{\bar{X}}=\cO_{\bar{X}}(-V_Z)$, then by (\ref{eqn:l1}) the divisor $l_1V_Z-V_F$ is effective with support $\bar{\varphi}^{-1}(Z)$. By $F\approx_{l_1}G$ we have the decomposition $\bar{\varphi}^*G=V_F+H_G$ in which $H_G$ consists of divisors not in $\bar{\varphi}^{-1}(Z)$, and moreover
\begin{align*}
&\cO_{\bar{X}}(-nV_F)(\cO_{\bar{X}}(-nH_F)+\cO_{\bar{X}}(-n(l_1V_Z-V_F))\\
&=\cO_{\bar{X}}(-nV_F)(\cO_{\bar{X}}(-nH_G)+\cO_{\bar{X}}(-n(l_1V_Z-V_F))
\end{align*}
for some $n$. Hence on the reduced divisor $\bar{\varphi}^{-1}(Z)$,
\begin{align}\label{eqn:support}
nH_F\cap\bar{\varphi}^{-1}(Z)=nH_G\cap\bar{\varphi}^{-1}(Z)
\end{align}
scheme-theoretically, and its support contains $\varphi'(E)$ by (\ref{eqn:resolutionG}). Thus there exists a prime divisor $\bar{E}$ on $\bar{X}$ with $\varphi'(E)\subset\bar{E}\subset\bar{\varphi}^{-1}(Z)$ and $\bar{E}\cap\Supp H_F\neq\emptyset$. $\bar{E}$ has $\ord_{\bar{E}}G=\ord_{\bar{E}}F>t^{-1}c$ by (\ref{eqn:t^-1c}), so $\ord_EF\ge\ord_{\bar{E}}F>t^{-1}c$, $\ord_EG\ge\ord_{\bar{E}}G>t^{-1}c$. 
\end{proof}

We obtain one inequality in Theorem \ref{thm:plt} as in Remark \ref{rmk:inequality}.

\begin{lemma}\label{lem:easy-plt}
If $(F,\fa)\approx_{l_1}(G,\fb)$, then $\mld_Z(F,\fa\cO_F)\ge\mld_{\nu^{-1}(Z)}(G^\nu,\fb\cO_{G^\nu})$.
\end{lemma}

\begin{proof}
We have the divisor $E_F\subset\bar{\varphi}^{-1}(Z)$ in (\ref{eqn:mldF}). $W:=\bar{F}\cap E_F$ is contained in the support of the locus (\ref{eqn:support}), whence $W\subset\Supp H_G\cap E_F$. This implies $W\subset\bar{G}\cap E_F$ for the strict transform $\bar{G}$ of $G$ by the s.n.c.\ property of $\bar{F}+E_F+\Supp(H_G-\bar{G})$. Moreover by (\ref{eqn:support}), $nW=n\bar{G}|_{E_F}$ as divisors on $E_F$ at the generic point $\eta_W$ of $W$. Hence $W=\bar{G}\cap E_F$ scheme-theoretically at $\eta_W$, and its strict transform $W'$ on $G'$ is defined. With (\ref{eqn:l1}) we obtain
\begin{align*}
\mld_{\nu^{-1}(Z)}(G^\nu,\fb\cO_{G^\nu})&\le a_{W'}(G^\nu,\fb\cO_{G^\nu})=a_{E_F}(X,G,\fb)=a_{E_F}(X,F,\fa)=c.
\end{align*}
\end{proof}

We shall prove the other inequality $\mld_{\nu^{-1}(Z)}(G^\nu,\fb\cO_{G^\nu})\ge c$ in Theorem \ref{thm:plt} by studying $E\in I_Z$ with $a_{E|_{G'}}(G^\nu,\fb\cO_{G^\nu})\le c$. We fix a prime divisor $E_Z$ on $\bar{X}$ such that $\bar{\varphi}(E_Z)=Z$, and apply Zariski's subspace theorem \cite[(10.6)]{A98} as in the proof of \cite[Lemma 3]{K07} to the natural map $\cO_{X,Z}\to\cO_{\bar{X},E_Z}$ and its specialisations, to fix an integer $l_2\ge l_1$ such that
\begin{align}\label{eqn:l2}
\bar{\varphi}_*\cO_{\bar{X}}(-l_2E_Z)\subset\cI_Z^{l_1}.
\end{align}

\begin{lemma}\label{lem:l2}
If $(F,\fa)\approx_{l_2}(G,\fb)$ and $E\in I_Z$ satisfies $a_{E|_{G'}}(G^\nu,\fb\cO_{G^\nu})\le c$, then
\begin{enumerate}
\item\label{itm:l2:ordJ'}
$\ord_E\tilde{\cJ}'_F=\ord_E\tilde{\cJ}'_G\le(rt)^{-1}c$.
\item\label{itm:l2:ordJ}
$\ord_E\tilde{\cJ}_{r,F}=\ord_E\tilde{\cJ}_{r,G}\le t^{-1}c$.
\item\label{itm:l2:ordX}
$\ord_E\cD'_X\le t^{-1}c$.
\item\label{itm:l2:ordab}
$\ord_E\fa=\ord_E\fb\le t^{-1}c$.
\end{enumerate}
\end{lemma}

\begin{proof}
(\ref{itm:l2:ordJ'}) \
We use explicit descriptions of $\tilde{\cJ}'_F,\tilde{\cJ}'_G$ in terms of Jacobian matrices. Embed $X$ into a smooth ambient space $A$ with local coordinates $x_1,\ldots,x_k$ and take $f,g\in\cO_A$ such that $f|_X,g|_X$ define $F,G$. By $F\approx_{l_2}G$, $f^n\cO_X+\cI_Z^{nl_2}=g^n\cO_X+\cI_Z^{nl_2}$ for some $n$. Note $f^n|_X\not\in\cI_Z^{nl_2}$ by $\ord_{E_Z}f|_X<l_1$ from (\ref{eqn:l1}). If we choose $u,v\in\cO_A$ so that $f^n-ug^n|_X,g^n-vf^n|_X\in\cI_Z^{nl_2}$, then $(1-uv)f^n|_X\in\cI_Z^{nl_2}$ so $uv$ should be a unit. We take an etale cover $\tilde{X}\to X$ by adding a function $y$ with $y^n=u$ to produce the factorisation $f^n-ug^n=\prod_i(f-\mu^iyg)$ with a primitive $n$-th root $\mu$ of unity, and discuss on the germ $\tilde{U}$ at some closed point of $\tilde{X}$. Set the prime divisor $\tilde{E}_Z:=E_Z\times_X\tilde{U}$ on $\tilde{\varphi}\colon\bar{X}\times_X\tilde{U}\to\tilde{U}$. Since $\prod_i(f-\mu^iyg)|_{\tilde{U}}\in\tilde{\varphi}_*\cO_{\bar{X}\times_X\tilde{U}}(-nl_2\tilde{E}_Z)$, with (\ref{eqn:l2}) there exists $i$ such that 
\begin{align*}
f-\mu^iyg|_{\tilde{U}}\in\tilde{\varphi}_*\cO_{\bar{X}\times_X\tilde{U}}(-l_2\tilde{E}_Z)=\bar{\varphi}_*\cO_{\bar{X}}(-l_2E_Z)\otimes_{\cO_X}\cO_{\tilde{U}}\subset\cI_Z^{l_1}\cO_{\tilde{U}}.
\end{align*}
$F\times_X\tilde{U},G\times_X\tilde{U}$ are given by $f|_{\tilde{U}},\mu^iyg|_{\tilde{U}}$. By the description of $\tilde{\cJ}'_F\cO_{\tilde{U}},\tilde{\cJ}'_G\cO_{\tilde{U}}$ in terms of Jacobian matrices, we have
\begin{align*}
\tilde{\cJ}'_F\cO_{\tilde{U}}+\cC=\tilde{\cJ}'_G\cO_{\tilde{U}}+\cC
\end{align*}
for $\cC:=\sum_j(\partial(f-\mu^iyg)/\partial x_j\cdot\cO_{\tilde{U}})\subset\cI_Z^{l_1-1}\cO_{\tilde{U}}$. By Lemma \ref{lem:l1}(\ref{itm:l1:ordJ'D}) and (\ref{eqn:estimate_l1}), for $\tilde{E}:=E\times_X\tilde{U}$
\begin{align*}
\ord_{\tilde{E}}\tilde{\cJ}'_F\cO_{\tilde{U}}&=\ord_E\tilde{\cJ}'_F\le(rt)^{-1}c<l_1-1,\\
\ord_{\tilde{E}}\tilde{\cJ}'_G\cO_{\tilde{U}}&=\ord_E\tilde{\cJ}'_G,
\end{align*}
which provide (\ref{itm:l2:ordJ'}).

(\ref{itm:l2:ordJ}) \
Lemma \ref{lem:l1} implies $\ord_E\tilde{\cJ}'_F{}^r\le t^{-1}c<\ord_EF,\ord_EG$. Thus (\ref{itm:l2:ordJ}) follows from (\ref{itm:l2:ordJ'}) and (\ref{eqn:defectX}) for $F,G$.

(\ref{itm:l2:ordX}) \
It follows from Lemma \ref{lem:l1}(\ref{itm:l1:ordJ'D}).

(\ref{itm:l2:ordab}) \
It follows from Lemma \ref{lem:l1}(\ref{itm:l1:ordJ'D}), (\ref{eqn:estimate_l1}) and Remark \ref{rmk:order}.
\end{proof}

We shall apply motivic integration by Kontsevich in \cite{Kn95} and Denef and Loeser in \cite{DL99} to transversal divisors. We fix notation following \cite[Section 3]{K08}. For a scheme $X$ of dimension $d$, we let $J_nX$ denote its \textit{jet scheme} of order $n$, $J_\infty X$ its \textit{arc space}, and set $\pi_n^X\colon J_\infty X\to J_nX$, $\pi_{nm}^X\colon J_mX\to J_nX$. One has the \textit{motivic measure} $\mu_X\colon\cB_X\to\widehat{\cM}$ from the family $\cB_X$ of \textit{measurable} subsets of $J_\infty X$ to an extension $\widehat{\cM}$ of the Grothendieck ring. $\cB_X$ is an extension of the family of stable subsets. A subset $S$ of $J_\infty X$ is said to be \textit{stable} at level $n$ if $\pi_n^X(S)$ is constructible, $S=(\pi_n^X)^{-1}(\pi_n^X(S))$, and $\pi_{m+1}^X(S)\to\pi_m^X(S)$ is piecewise trivial with fibres $\bA^d$ for $m\ge n$. $S$ has measure
\begin{align*}
\mu_X(S)=[\pi_n^X(S)]\bL^{-(n+1)d}
\end{align*}
with $\bL=[\bA^1]$.

For a morphism $\varphi\colon X\to Y$, we write $\varphi_n\colon J_nX\to J_nY$, $\varphi_\infty\colon J_\infty X\to J_\infty Y$ for the induced morphisms. For a closed subset $Z$, we let $J_nX|_Z,J_\infty X|_Z$ denote the inverse images of $Z$ by $J_nX,J_\infty X\to X$. Finally for an $\bR$-ideal sheaf $\fa$, the \textit{order} $\ord_\fa\gamma$ along $\fa$ is defined for $\gamma\in J_\infty X$. The notion of $\ord_\cI\gamma_n$ for an ideal sheaf $\cI$ makes sense even for $\gamma_n\in J_nX$ as long as $\ord_\cI\gamma_n\le n$.

Back to the theorem, we fix an expression
\begin{align*}
\fa=\fa_1^{r_1}\cdots\fa_k^{r_k}.
\end{align*}
We fix an integer $c_1$ such that
\begin{align}\label{eqn:c1}
c_1\ge t^{-1}c,\qquad c_1\ge(r_it)^{-1}c
\end{align}
for all $i$. Applying Greenberg's result \cite{G66} to $F$, one can find $c_2\ge c_1$ such that
\begin{align}\label{eqn:Greenberg}
\pi_{c_1c_2}^F(J_{c_2}F)=\pi_{c_1}^F(J_\infty F).
\end{align}
We take an integer $l_3\ge l_2$ such that
\begin{align}\label{eqn:l3}
l_3>c_2.
\end{align}

From now on we fix an arbitrary $E\in I_Z$ for $(G,\fb)\approx_{l_3}(F,\fa)$ such that
\begin{align}\label{eqn:c}
a_{E|_{G'}}(G^\nu,\fb\cO_{G^\nu})\le c,
\end{align}
and will derive the opposite inequality $a_{E|_{G'}}(G^\nu,\fb\cO_{G^\nu})\ge c$. To avoid confusion we set $\psi:=\varphi|_{G'}\colon G'\to G$. By blowing up $X'$ further, we may assume that $E'|_{G'}$ is $\psi$-exceptional for all $E'\in I\setminus\{E\}$ with $E|_{G'}\cap E'|_{G'}\neq\emptyset$. Take the subset $T'$ of $J_\infty G'$ which consists of all arcs $\gamma$ such that
\begin{align*}
\ord_{E'|_{G'}}\gamma=
\begin{cases}
1&\text{if $E'=E$,}\\
0&\text{if $E'\in I\setminus\{E\}$, $E'|_{G'}\cap E|_{G'}\neq\emptyset$.}
\end{cases}
\end{align*}
$T'$ is stable at level one. Set $T:=\psi_\infty(T')\subset J_\infty G$, $T'_n:=\pi_n^{G'}(T')\subset J_nG'$ and $T_n:=\pi_n^G(T)=\psi_n(T'_n)\subset J_nG$ as
\begin{align*}
\xymatrix{
J_\infty G'\ar@{}|{\supset}[r]\ar@{>}[d]&
T'\ar@{>>}[r]^{\pi_n^{G'}}\ar@{>>}[d]_{\psi_\infty}&
T'_n\ar@{}|{\subset}[r]\ar@{>>}[d]^{\psi_n}&
J_nG'\ar@{>}[d]\\
J_\infty G\ar@{}|{\supset}[r]&
T\ar@{>>}[r]^{\pi_n^G}&
T_n\ar@{}|{\subset}[r]&
{}\phantom{.}J_nG.
}
\end{align*}
One can regard $J_nF,J_nG\subset J_nX$. Then $F\approx_{l_3}G$ implies $J_{c_2}F|_Z=J_{c_2}G|_Z$ by (\ref{eqn:l3}). Hence by (\ref{eqn:Greenberg})
\begin{align*}
T_{c_1}\subset\pi_{c_1c_2}^G(J_{c_2}G|_Z)=\pi_{c_1c_2}^F(J_{c_2}F|_Z)=\pi_{c_1}^F(J_\infty F|_Z).
\end{align*}
Thus if we set
\begin{align*}
S:=(\pi_{c_1}^F)^{-1}(T_{c_1})\subset J_\infty F
\end{align*}
and $S_n:=\pi_n^F(S)\subset J_nF$, then $S_{c_1}=T_{c_1}$ as
\begin{align}\label{eqn:S=T}
\xymatrix{
J_\infty F\supset S\ar@{>>}[r]^(0.64){\pi_n^F}&
S_n\ar@{>>}[r]^(0.39){\pi_{c_1n}^F}&
S_{c_1}=T_{c_1}.
}
\end{align}

We translate Lemma \ref{lem:l2} into the language of arcs.

\begin{lemma}\label{lem:ST}
\begin{enumerate}
\item\label{itm:ST:ordJ'}
On $S,T$, $\ord_{\tilde{\cJ}'_F}=\ord_{\tilde{\cJ}'_G}$ and takes constant $\ord_E\tilde{\cJ}'_F=\ord_E\tilde{\cJ}'_G\le c_1$.
\item\label{itm:ST:ordJ}
On $S,T$, $\ord_{\tilde{\cJ}_{r,F}}=\ord_{\tilde{\cJ}_{r,G}}$ and takes constant $\ord_E\tilde{\cJ}_{r,F}=\ord_E\tilde{\cJ}_{r,G}\le c_1$.
\item\label{itm:ST:ordD}
On $S,T$, $\ord_{\cD'_X}$ takes constant $\ord_E\cD'_X\le c_1$.
\item\label{itm:ST:ordab}
On $T$, $\ord_\fa=\ord_\fb$ and takes constant $\ord_E\fa=\ord_E\fb\le c_1$. On $S$, $\ord_\fa$ takes constant $\ord_E\fa=\ord_E\fb$.
\end{enumerate}
\end{lemma}

\begin{proof}
It is obvious by Lemma \ref{lem:l2}, (\ref{eqn:c1}) and the construction of $T'$. Note $\ord_E\fa_i\le r_i^{-1}\ord_E\fa\le c_1$.
\end{proof}

Let $\cJ_\psi$ be the image of the natural map $\psi^*\Omega_G^{d-1}\otimes\omega_{G'}^{-1}\to\cO_{G'}$. By definition we obtain the equality
\begin{align*}
\cJ_\psi^r=\tilde{\cJ}_{r,G}\cO_{G'}\big(-r\sum_{E'\in I}(a_{E'|_{G'}}(G^\nu)-1)E'|_{G'})\big).
\end{align*}
Hence $\cJ_\psi$ is resolved on $G'$, and on $T'$ the order along $\cJ_\psi$ takes constant 
\begin{align*}
e:=\ord_{E|_{G'}}\cJ_\psi=r^{-1}\ord_E\tilde{\cJ}_{r,G}+a_{E|_{G'}}(G^\nu)-1.
\end{align*}

We use the following form of \cite[Lemma 4.1]{DL99} to estimate $\mu_F(S)$.

\begin{proposition}\label{prp:fibre}
Let $X$ be a reduced scheme of pure dimension, and $L_n^X$ the locus of $J_\infty X$ on which the orders along the Jacobian ideal sheaf $\cJ'_X$ and the l.c.i.\ defect ideal sheaf $\cD'_X$ are at most $n$. Then $L_n^X$ is stable at level $n$.
\end{proposition}

\begin{proof}
For a l.c.i.\ scheme, the proposition follows from the proof of \cite[Lemma 4.1]{DL99} directly. Note that the l.c.i.\ defect ideal sheaf of a l.c.i.\ scheme is trivial.

For general $X$, we fix a jet $\gamma_n\in \pi_n^X(L_n^X)$. By the definitions of $\cJ'_X,\cD'_X$, one can embed $X$ into a l.c.i.\ scheme $Y=X\cup C^Y$ as (\ref{eqn:lci}) so that on a neighbourhood $U_{\gamma_n}$ of $\gamma_n$ in $J_nY$, $\ord_{\cJ'_Y}\le\ord_{\cJ'_X}(\gamma_n)$ and $\ord_{\cC_{X/Y}}\le\ord_{\cD'_X}(\gamma_n)$ for the Jacobian $\cJ'_Y$ and the conductor $\cC_{X/Y}$. Then $(\pi_n^X)^{-1}(U_{\gamma_n})\subset L_n^X$ and $(\pi_n^Y)^{-1}(U_{\gamma_n})\subset L_n^Y$. By $\cC_{X/Y}\cI_{X/Y}=0$ for the ideal sheaf $\cI_{X/Y}$ of $X$ on $Y$, we have $J_\infty Y\setminus(\ord_{\cC_{X/Y}})^{-1}(\infty)\subset J_\infty X$. Hence $(\pi_n^X)^{-1}(U_{\gamma_n})=(\pi_n^Y)^{-1}(U_{\gamma_n})$, and the statement is reduced to that of the l.c.i.\ scheme $Y$.
\end{proof}

\begin{lemma}\label{lem:dimS}
$\mu_F(S)=\mu_G(T)=\mu_{G'}(T')\bL^{-e}$.
\end{lemma}

\begin{proof}
We apply Proposition \ref{prp:fibre} to $S\subset L_{c_1}^F$, $T\subset L_{c_1}^G$ by Lemma \ref{lem:ST}(\ref{itm:ST:ordJ'}), (\ref{itm:ST:ordD}) and (\ref{eqn:defect}), to obtain their stabilities at level $c_1$ and by $S_{c_1}=T_{c_1}$ in (\ref{eqn:S=T}) 
\begin{align*}
\mu_F(S)=\mu_G(T).
\end{align*}

By \cite[Lemma 3.4]{DL99} for $T\subset\cL^{(c_1)}(G)$ with notation in \cite{DL99}, there exists $n\ge c_1,e,1$ such that $\ord_{\cJ_\psi}$ takes constant $e$ on $\psi_n^{-1}(T_n)$, and that $\psi_n^{-1}(T_n)\to T_n$ is piecewise trivial with fibres $\bA^e$. If the equality $T'_n=\psi_n^{-1}(T_n)$ holds, then
\begin{align*}
\mu_G(T)=[T_n]\bL^{-(n+1)(d-1)}=[T'_n]\bL^{-(n+1)(d-1)-e}=\mu_{G'}(T')\bL^{-e}.
\end{align*}
Thus it suffices to prove $\psi_n^{-1}(T_n)\subset T'_n$.

Take a variety $U_n$ dense in $T_n$ such that $\psi_n^{-1}(U_n)$ is irreducible. The closure $C_n$ of $\psi_n^{-1}(U_n)$ in $J_nG'$ contains the closure $J_nG'|_{E|_{G'}}$ of $T'_n$, which is a prime divisor. Thus $C_n=J_nG'|_{E|_{G'}}$ by the irreducibility of $C_n$, so the image of the restricted morphism $\chi_n\colon J_nG'|_{E|_{G'}}\to J_nG$ contains $T_n$. Its fibre $\chi_n^{-1}(t)$ at $t\in T_n$ has dimension at least $e$ and is contained in $\psi_n^{-1}(t)\simeq\bA^e$. Hence $\chi_n^{-1}(t)=\psi_n^{-1}(t)$ as $\chi_n^{-1}(t)$ is closed. This means $\psi_n^{-1}(T_n)\subset J_nG'|_{E|_{G'}}$.

Consider on $\psi_n^{-1}(T_n)$ the constant function
\begin{align*} 
e=\ord_{\cJ_\psi}=\sum_{E'\in I}(\ord_{E'|_{G'}}\cJ_\psi)\cdot\ord_{E'|_{G'}}.
\end{align*}
Note that
\begin{align*}
\ord_{E|_{G'}}\cJ_\psi=e,\qquad\ord_{E'|_{G'}}\cJ_\psi>0\text{\ for $E'\in I\setminus\{E\}$, $E'|_{G'}\cap E|_{G'}\neq\emptyset$},
\end{align*}
because such $E'|_{G'}$ is $\psi$-exceptional and $\cJ_\psi$ vanishes on the support of $\Omega_{G'/G}$. Moreover $\ord_{E|_{G'}}$ is positive on $\psi_n^{-1}(T_n)\subset J_nG'|_{E|_{G'}}$. Hence $\psi_n^{-1}(T_n)\subset T'_n$ by the definition of $T'$.
\end{proof}

\begin{remark}
We need only the inequality $\dim\mu_F(S)\ge\dim\mu_{G'}(T')\bL^{-e}$ for the proof of Theorem \ref{thm:plt}.
\end{remark}

We shall complete the proof by using the below description of $c=\mld_Z(F,\fa\cO_F)$ in terms of motivic integration by \cite{EMY03}; see also \cite[Remark 3.3]{K08}.
\begin{align}\label{eqn:mld}
c=-\dim\int_{J_\infty F|_Z}\bL^{r^{-1}\ord_{\tilde{\cJ}_{r,F}}+\ord_\fa}d\mu_F.
\end{align}

\begin{lemma}\label{lem:hard-plt}
If $(F,\fa)\approx_{l_3}(G,\fb)$, then $\mld_Z(F,\fa\cO_F)\le\mld_{\nu^{-1}(Z)}(G^\nu,\fb\cO_{G^\nu})$.
\end{lemma}

\begin{proof}
We have fixed an arbitrary $E\in I_Z$ which satisfies (\ref{eqn:c}). By Lemma \ref{lem:ST}(\ref{itm:ST:ordJ}), (\ref{itm:ST:ordab}), $\ord_{\tilde{\cJ}_{r,F}},\ord_\fa$ take constants $\ord_E\tilde{\cJ}_{r,G},\ord_E\fb$ on $S$. Thus with Lemma \ref{lem:dimS},
\begin{align*}
\int_S\bL^{r^{-1}\ord_{\tilde{\cJ}_{r,F}}+\ord_\fa}d\mu_F&=\mu_F(S)\bL^{r^{-1}\ord_E\tilde{\cJ}_{r,G}+\ord_E\fb}\\
&=\mu_{G'}(T')\bL^{r^{-1}\ord_E\tilde{\cJ}_{r,G}+\ord_E\fb-e},
\end{align*}
and
\begin{align*}
\dim\int_{J_\infty F|_Z}\bL^{r^{-1}\ord_{\tilde{\cJ}_{r,F}}+\ord_\fa}d\mu_F&\ge\dim\int_S\bL^{r^{-1}\ord_{\tilde{\cJ}_{r,F}}+\ord_\fa}d\mu_F\\
&=-1+r^{-1}\ord_E\tilde{\cJ}_{r,G}+\ord_E\fb-e\\
&=-a_{E|_{G'}}(G^\nu)+\ord_E\fb\\
&=-a_{E|_{G'}}(G^\nu,\fb\cO_{G^\nu}).
\end{align*}
Hence $a_{E|_{G'}}(G^\nu,\fb\cO_{G^\nu})\ge c$ by (\ref{eqn:mld}), which proves the lemma.
\end{proof}

Theorem \ref{thm:plt} is therefore proved.

{\small
\begin{acknowledgements}
This research has originated in the problem of Professor M. Musta\c{t}\u{a}. I was asked it at the workshop at MSRI in 2007, and discussed with him during my visit at University of Michigan in 2009. I really appreciate his approval of introducing his problem and observation. I am also grateful to him for valuable discussions and cordial hospitality. Partial support was provided by Grant-in-Aid for Young Scientists (A) 20684002.
\end{acknowledgements}
}

\bibliographystyle{amsplain}

\end{document}